\documentclass[a4paper,12pt,reqno]{amsart}
\usepackage{amsfonts}
\usepackage{amsmath}
\usepackage{amssymb}
\usepackage[a4paper]{geometry}
\usepackage{mathrsfs}
\usepackage{hyperref}
\renewcommand\eqref[1]{(\ref{#1})} %Need with hyperref
\numberwithin{equation}{section}

\newtheorem{theorem}{Theorem}[section]
\newtheorem{proposition}[theorem]{Proposition}

\newtheorem{corollary}[theorem]{Corollary}

\theoremstyle{definition}

\newtheorem{remark}[theorem]{Remark}

\numberwithin{equation}{section}
\newcommand {\R}{{\mathbb R}}

\begin{document}
%%%%%%%%%%%%%%%%%%%%%%%%%%%%%%%%%%%%%%%%%%%%%%%%%%%%%%%%%%%%%%%%%%%%%%%%%%%%%%%%%%%%%%%

%%%%%%%%%%%%%%%%%%%%%%%%%%%%%%%%%%%%%%%%%%%%%%%%%%%%%%%%%%%%%%%%%%%%%%%%%%%%%%%%%%%%%%%

%%%%%%%%%%%%%%   Top matter %%%%%%%%%%%%%%%%%%%%%%%%%%%%%%%%%%%%%%%%%%%%%%%%%%%%%%%%%%%

\title{Caffarelli--Kohn--Nirenberg inequalities on Lie groups of polynomial growth}

             % Linebreaks by \\, leave blank if not needed

\translator{}
             % Linebreaks by \\, leave blank if not needed

\dedicatory{}
            % Linebreaks by \\, leave blank if not needed

\author[Chokri~Yacoub]{Chokri~Yacoub}
             % Comment out if not needed

\begin{abstract}
In the setting of a Lie group of polynomial volume growth, we derive  inequalities of Caffarelli--Kohn--Nirenberg type,
where the weights involved are powers of the Carnot--Caratheodory distance associated with a 
fixed system of vector fields which satisfy the H\"ormander condition.

The use of weak $L^p$ spaces is crucial in our proofs and we formulate these inequalities 
within the framework of $L^{p,q}$ Lorentz spaces (a scale of (quasi)-Banach spaces  which extend the more
classical $L^p$ Lebesgue spaces) thereby obtaining a refinement of, for instance, Sobolev and Hardy--Sobolev inequalities.

\end{abstract}
             % Comment out if not needed

\date{}
             % Footnote on page 1, first item
             % Leave blank if not needed

             % Footnote on page 1, second item
             % Enter as {Primary primarynumber; Secondary secundarynumbers}
             % Leave blank if not needed

             % Footnote on page 1, third item
             % Enter the keywords, separated by commas, no period at the end.
             % Leave blank if not needed

\thanks{}
             % Footnote on page 1, fourth item.
             % Do not specify linebreaks
             % Leave blank if not needed

\address{Universit\'e de Monastir\\
Facult\'e des Sciences 
de Monastir\\
D\'epartement de Math\'ematiques\\
5019 Monastir\\
Tunisie}

             % Comment out if not needed

%\curraddress{}
             % Linebreaks by \\, leave blank if not needed.
             % Command is not recognized????? (October 8, 2000)

\email{chokri.yacoub@fsm.rnu.tn}
             % Comment out if not needed

\urladdr{}
             % Leave blank if not needed

\maketitle

%%%%%%%%%%%%%%   Start of the text %%%%%%%%%%%%%%%%%%%%%%%%%%%%%%%%%%%%%%%%%%%%%%%%%%%%
\section{Introduction}\label{sec:intro}

The classical Hardy inequality in the Euclidean space $\R^n$, $n\geq 3$, asserts that for all compactly supported smooth functions $f$,

\begin{equation}\label{eq:intro}
\int_{\R^n}\frac{\left|f(x)\right|^2}{\left|x\right|^2} dx \leq \left(\frac{2}{n-2}\right)^2\int_{\R^n}\left|\nabla f(x)\right|^2dx.
\end{equation}

If one is not interested in the best constant, one can derive (via the spherically symmetric rearrangement ) another mathematical formulation 
of uncertainty principle, namely the Sobolev inequality

\begin{equation}\label{eq:intro1}
\left(\int_{\R^n}\left|f(x)\right|^{\frac{2n}{n-2}}dx\right)^{\frac{n-2}{2n}} \leq c(n) \left(\int_{\R^n}\left|\nabla f(x)\right|^2dx \right)^{1/2}.
\end{equation}

Now, let $0\leq s\leq 2$ and define the critical exponent relative to $s$ by $2^\ast(s)=2\frac{n-s}{n-2}$. If we interpolate between \eqref{eq:intro} and \eqref{eq:intro1} then we obtain the Hardy-Sobolev inequality:

\begin{equation}\label{eq:intro3}
\left(\int_{\R^n}\frac{\left|f(x)\right|^{2^\ast(s)}}{\left|x\right|^s} dx \right)^{\frac{1}{2^\ast(s)}} \leq c(n,s) \left(\int_{\R^n}\left|\nabla f(x)\right|^2dx \right)^{1/2}.
\end{equation}

Indeed, observe that $2^\ast(s) = 2 \theta +(1-\theta)\frac{2n}{n-2}$, with $\theta =\frac{s}{2}$. Then by H\"older's 
inequality we obtain

\begin{align*}
\int_{\R^n}\frac{\left|f(x)\right|^{2^\ast(s)}}{\left|x\right|^s} dx&=\int_{\R^n}\frac{\left|f(x)\right|^{2 \theta +(1-\theta)\frac{2n}{n-2}}}{\left|x\right|^s} dx\\
&\leq \left(\int_{\R^n}\frac{\left|f(x)\right|^2}{\left|x\right|^2} dx\right)^\theta \left(\int_{\R^n}\left|f(x)\right|^{\frac{2n}{n-2}}\right)^{1-\theta}\notag\\
&\leq c'(n) \left(\int_{\R^n}\left|\nabla f(x)\right|^2dx \right)^{\frac{2^\ast(s)}{2}}
\end{align*}

by \eqref{eq:intro} and \eqref{eq:intro1}.
\vskip .2cm

The Caffarelli--Kohn--Nirenberg inequalities \cite{ckn} are a family of functional inequalities, with spherical weight, that extend and unify many classical ones like \eqref{eq:intro}, \eqref{eq:intro1}, \eqref{eq:intro3} etc. More precisely:
let $p$ $\geq$ $1$, $q$ $\geq$ $1$, $r$ $>0$; $\alpha$, $\beta$, $\gamma$; and $0$ $<$ $a$ $\leq$ $1$ fixed real numbers satisfying

\begin{equation} \label{eq:basis}
-\gamma< \frac{n}{r},\,\,-\alpha< \frac{n}{p},\,\,-\beta< \frac{n}{q}
\end{equation}

\begin{equation} \label{eq:basis}
\gamma+\frac{n}{r}=a(\alpha-1+\frac{n}{p})+(1-a)(\beta+\frac{n}{q}).
\end{equation}

Denote with $\Delta$ the quantity

$$\Delta=-\gamma+a\alpha+(1-a)\beta.$$

Then the Caffarelli--Kohn--Nirenberg inequalities, as formulated by D'Ancona and Luca' in \cite{luca}, states that there exists a positive constant $c$ such that for smooth compactly supported functions $f$ on the Euclidean n-space,

\begin{equation}\label{ckn}
\left(\int_{\R^n} \left|x\right|^{\gamma r}\left|f(x)\right|^r dx\right)^{1/r} \leq c \left(\int_{\R^n} \left|x\right|^{\alpha p}\left|\nabla f(x)\right|^p dx\right)^{a/p} \left(\int_{\R^n} \left|x\right|^{\beta q}\left|f(x)\right|^q dx\right)^{(1-a)/q}
\end{equation}

if and only if 

i) $\Delta \geq 0$

ii) $\Delta\leq a$ when $\frac{1}{r}+\frac{\gamma}{n}=\frac{1}{p}+\frac{\alpha-1}{n}$.

\vskip .2cm

A natural step further is to extend the above Caffarelli--Kohn--Nirenberg inequalities ( in short CKN) to the case of Lie groups or Riemannian manifolds.
This note is a contribution to the study of CKN inequalities on a Lie group $G$. In this direction, we list only some references that share a similar 
theme as ours. In \cite{gri}, Grillo obtained some global weighted Sobolev inequalities (see his Theorem 2.2), where the weights are powers of an  intrinsic metric associated to a system of locally Lipschitz vector fields verifying certain conditions fulfilled, in particular, by H\"ormander's vector fields. As he noticed, his approach relies crucially upon the Sawyer-Wheeden condition for weighted inequality (see \cite{sa} or \cite{p}).
A class of CKN inequalities was established in \cite{chi} in the setting of an important subclass of two-step stratified nilpotent Lie groups, namely 
groups of Heisenberg-type. In \cite{loh}, Lohou\'e obtained an $L^p$ Hardy-Sobolev type inequalities in different frameworks such as the homogeneous, polynomial growth or nonamenable Lie groups. We also mention \cite{cow} (especially Section 6) and the recent \cite{ru1} and \cite{ru2}.

\section{Statement of the results and method of the proofs}\label{sec:result}

To present our first result, we need to introduce some terminology and for any exact definition, please refer to Section ~\ref{sec:basic}.
As annouced in the abstract, the aim of this paper is to give a version of CKN inequalities in various Lorentz (quasi)-norms on a  Lie group of polynomial growth.
The Euclidean distance which appears in \eqref{ckn} will be replaced by the Carnot-Caratheodory distance $\rho$ induced by a fixed 
 H\"ormander system $\mathbb{X}=\{X_{1}, \dots, X_{k}\}$ of left invariant vector fields on $G$. For short we will write $\rho (x)=\rho(e,x)$
 where $e$ is the  identity element of $G$. The local dimension of $(G,\mathbb{X})$ will be denoted by $d$. Contrary to the later, the dimension at infinity, noticed here by $D$, is independant of a particular choice of a system $\mathbb{X}$. We will assume that $d<D$. This is the case if for instance $G$ is a not stratified simply connected nilpotent Lie group. But we would like to note that all our results are true (and even easier to handle) in the context of a homogeneous nilpotent Lie group, where the growth of the quasi-balls (associated to any homogeneous quasi-norm on
$G$) is governed by a single integer that is the homogeneous dimension of the group.

\begin{theorem}\label{th1}
Let $n \in \left[d,D\right]$, $1<p<d$, $r,q\geq 1$; $ \beta, \gamma$; and $0<a\leq 1$ fixed real numbers satisfying
\begin{equation} \label{eq:balance}
\frac{1}{r}+\frac{\gamma}{n}=a(\frac{1}{p}-\frac{1}{n})+(1-a)(\frac{1}{q}+\frac{\beta}{n}).
\end{equation}
Let also $1\leq q_1,q_2,q_3 \leq \infty$ such that $\frac{1}{q_1}=\frac{a}{q_2}+\frac{1-a}{q_3}$. Denote the quantity
\begin{equation} \label{delta}
\delta = -\gamma +(1-a)\beta.
\end{equation}
Assume that
\begin{equation} \label{cond}
0\leq\delta\leq a,
\end{equation}
then
\begin{equation} \label{eq1}
\left\|\rho^\gamma f\right\|_{r,q_1} \leq C \left\|\left|X f\right|\right\|_{p,q_2}^a\left\|\rho^\beta f\right\|_{q,q_3}^{1-a}
\end{equation}
for all $f \in C^\infty _c (G)$.
\end{theorem}
Here $\left|Xf\right|^2 = \sum_{i=1}^k (X_{i}f)^2$. 
Let us make some remarks concerning the approach used in the proof of \eqref{eq1}. As is well known, in the Euclidean space $\R^n$, the most obvious way to obtain Sobolev type inequalities is to use some representation formula expressing a smooth compactly supported function $f$ as a fractional integral of its gradient. We will follow the same path and, unlike Lohou\'e in \cite{loh}, this approach avoid us using the $L^p$ boundedness of the Riesz transform on a Lie group of polynomial growth which is a profound result of Alexopoulos \cite{alex}.
Our starting point is a pointwise representation (see \cite{sc}):

\begin{equation}\label{eq:pr}
f(x) = \sum_{i = 1}^{k} N_{i} \ast X_{i}f(x)
\end{equation}
for any smooth, compactly supported function $f$ in $G$. Here the convolution operation on $G$ is defined by
$$f\ast g (x) = \int_G f(y)\,g(y^{-1} x)\,dy = \int_G f(x y^{-1})\,g(y)\,dy$$
and
\begin{equation}\label{eq:noyau}
N_{i}(x) = c \int_{0}^{+\infty} X_{i}h_{t}(x^{-1})\,dt
\end {equation}
where $h_{t}$ denotes the heat kernel associated to the sublaplacian $\sum_{i = 1}^{k} X_{i}^{2}$.
Moreover by [\cite{sc}, Lemma 2], the kernel $N_{i}$ verifies the following estimate

\begin{equation}\label{eq:estimate}
\left|N_{i}(x)\right| \lesssim \frac{\rho (x)}{V(\rho (x))}
\end{equation}
where $V(r)$ denotes the Haar measure of any ball of rayon $r$ with respect to the Carnot-Caratheodory distance.
Let us denote by $N(x)$ the right-hand side of \eqref{eq:estimate} and define the integral operator

\begin{align}
T(f)(x)&=N\ast f(x) \notag
\\&=\int_G N(y)f(y^{-1}x)\,dy
\end{align}
so that by \eqref{eq:pr} and \eqref{eq:estimate}
\begin{equation}\label{eq:iné}
\left|f(x)\right| \lesssim \sum_{i = 1}^{k}\,T\left( \left|X_{i}f\right| \right)\,(x).
\end {equation}
Using\eqref {eq:iné}, we obtain inequalities in that order: Sobolev, Hardy, then Hardy--Sobolev and finally \eqref{eq1}.

Certainly, the reader noticed that the first term in the right-hand side of \eqref{eq1} is unweighted. To obtain the full weighted CKN inequalities (i.e. all the three terms are weighted), we are led to study weighted $L^{p,\,q}$ estimates for the convolution operator $T$ of type
\[
\left\|vT(uf)\right\|_{s,q_{3}} \lesssim \left\|f\right\|_{p,q_{2}},
\]
where $u$ and $v$ are two nonnonnegative locally integrable functions. This is the content of the Proposition \ref{pro} , where we give sufficient conditions
 on $u$ and $v$ to obtain such inequality.
Finally, applying Proposition \ref{pro} to a particular choice of weights, we obtain the following

\begin{theorem}\label{th2}
Let $p$, $q$, $r$ $>$ $1$; $\alpha$ $>$ $0$, $ \beta$, $\gamma$; and $0$ $<$ $a$ $\leq$ $1$ fixed real numbers satisfying
\begin{equation} \label{li}
\gamma < \frac{d}{r},\,\,\,\, \beta < \frac{d}{q},\,\,\,\,1 - \frac{d}{p} < \alpha < \frac{d}{p'}
\end{equation}
\begin{equation} \label{eq:balance1}
\gamma-\frac{d}{r} = a\left( 1 -\alpha - \frac{d}{p}\right)+(1-a)\left( \beta - \frac{d}{q}\right).
\end{equation}
Let also $1\leq q_1,q_2,q_3 \leq \infty$ such that $\frac{1}{q_1}=\frac{a}{q_2}+\frac{1-a}{q_3}$. Denote the quantity
\begin{equation} \label{Delta}
\Delta = \gamma + a\alpha-(1-a)\beta.
\end{equation}
Assume that
\begin{equation} \label{cond}
0\leq\Delta\leq a,
\end{equation}
then
\begin{equation} \label{eq2}
\left\|\rho^{-\gamma} f\right\|_{r,q_1} \leq C \left\|\rho^\alpha \left|X f\right|\right\|_{p,q_2}^a\left\|\rho^{-\beta} f\right\|_{q,q_3}^{1-a}
\end{equation}
for all $f \in C^\infty _c (G)$.
\end{theorem}
The rest of this paper is organized as follows. In Section ~\ref{sec:basic} we introduce a few basic facts about  Lie groups of polynomial growth, and the necessary definitions and properties of Lorentz spaces. In Section ~\ref{sec:ckn}, after deriving Sobolev, Hardy and Hardy--Sobolev inequalities in Lorentz spaces, we prove Theorem \ref{th1}. In Section ~\ref{sec:ckn1}, after proving Proposition \ref{pro}, we use it (more precisely we use its Corollary \ref{cor}) to prove Theorem \ref{th2}.

Finally, let us mention that we adopt the notations $A \lesssim B$ to design $A\leq c B$ for some constant $c$ $>$ $0$ and $A \sim B$ for $c_1 B \leq A \leq c_2 B$. More precisely, in our context, $A$ and $B$ will represent (mainly) Lorentz norms of a function $f$ and the constant $c$, which varies from line to line, will depend on Lorentz indices only.

\noindent {\bf Acknowledgements.} The author would like to thank Sami Mustapha for pointing out the reference \cite{tycho}.
\section{Basic facts about Lie groups of  polynomial growth and Lorentz spaces}\label{sec:basic}

{\bf Lie groups of polynomial volume growth.} We recall some basic definitions and facts
about Lie groups of polynomial growth,
for which we refer for instance to
\cite{livre}, \cite{dungey}, 
\cite{alex}, \cite{gui} and \cite{martini}
for a deeper insight.

Let $G$ be a non-compact connected real Lie group and
let $dx$ be a left invariant Haar measure on $G$.
If $A\subset G$ is a Borel set,
we will denote by $\left|A\right|$ its Haar measure.
Let $U$ be a fixed compact neighborhood of the identity
element $e$ of $G$ and $U^j = \left\{x_{1} \cdots x_{j}:x_{1}, \dots, x_{j} \in U\right\}$.
By \cite{gui} we have the following dichotomy: either for some integer $D$
$$\left|U^j \right| \sim j^{D}$$
for all positive integers $j$, 
and we say that $G$ has polynomial growth of order $D$, or
$$e^{\alpha j} \lesssim \left|U^j\right| \lesssim e^{\beta j}$$
for some $\beta$ $\geq$ $\alpha$ $>$ $0$
and we say that $G$ has exponential growth. In the first case the group $G$ is known to be unimodular \cite{gui}.
Now by [\cite{livre}, Proposition III.4.2], we know that
connected distances on $G$ are always equivalent at infinity. 
Moreover, the Haar measure of large balls $B(r)$, with respect to any left-invariant
distance, grows like $D$-power of the radius:
\begin{equation}\label{D}
V(r) := \left|B(r)\right| \sim r^{D},\,\,\,\,\,\,r \geq 1.
\end{equation}
(For the definition of a connected distance, see [\cite{livre}, p. 40 
and Remark III.4.3].)
Denote by $\mathcal{G}$ the Lie algebra of $G$.
Let $\mathbb{X}=\{X_{1}, \dots, X_{k}\}$ be a H\"ormander system of 
left invariant vector fields on $G$, 
which means that $X_{1} \dots, X_{k}$ generate $\mathcal{G}$ as a Lie algebra. The Carnot-Caratheodory distance induced by $\mathbb{X}$ and denoted by $\rho$, is an example
of a such connected distance. (For its precise definition we refer also to \cite{livre}.) With respect to this distance, Haar measure of small balls $B(r)$ is governed by some positive integer $d$:
\begin{equation}\label{d}
V(r) = \left|B(r)\right| \sim r^{d},\,\,\,\,\,\,0 <r < 1.
\end{equation}
\vskip 0.5cm
\noindent {\bf Lorentz spaces.} As references we cite \cite{hunt}, \cite{o'}, or \cite{be}.
For a measurable function $f$,
$$\mu_f(\lambda)=\left|\left\{x\in G:\left|f(x)\right|>\lambda\right\}\right|$$
is the distribution function of $f$.
The decreasing rearrangement of $f$ is given by
$$f^\ast(t)=\inf\left\{\lambda>0:\mu_f(\lambda)\leq t\right\}.$$
The average function of $f^{\ast}$ is defined by
$$f^{\ast\ast}(t)=\frac{1}{t}\int_0^t f^{\ast}(s)\,ds$$
and it is clear that
\begin{equation}\label{1}
f^\ast(t) \leq f^{\ast\ast}(t).
\end{equation}
For the product, we have
\begin{equation}\label{2}
\left( f g\right)^\ast \left( t_1 + t_2\right) \leq f^\ast\left( t_1\right) g^\ast\left( t_2\right),
\end{equation}
and
\begin{equation}\label{3}
\int_0^t\,(f g)^\ast (s)\,ds \leq \int_0^t\,f^\ast (s) g^\ast (s)\,ds.
\end{equation}

\noindent Let $1$ $\leq$ $p,\ q$ $\leq$ $\infty$. The Lorentz space $L^{p,\,q}$ = $L^{p,\,q}(G)$ is defined as the collection of all measurable function $f$ 
such that $\left\|f\right\|_{p,\,q}$ is finite, where
$$\left\|f\right\|_{p,\,q} = \left( \int_0^\infty \left( t^{1/p-1/q}\,f^\ast (t) \right)^q\,dt\right)^{1/q},\,\,\,\,if 1 \leq q < \infty$$
and
$$\left\|f\right\|_{p,\,\infty} = \sup_{t > 0} t^{1/p} f^\ast (t).$$

\noindent We have 
\begin{equation}\label{id}
\left\|f\right\|_{p, p} = \left\|f\right\|_p
\end{equation}
and with respect to the second indice we have: if $q_1$ $\leq$ $q_2$, then
\begin{equation}\label{in}
\left\|f\right\|_{p,q_2} \leq \left\|f\right\|_{p,q_1},
\end{equation}
hence for $q_1$ $<$ $p$ $<$ $q_2$ we have the following inclusions
\begin{equation}\label{inc}
L^{p,\,q_1} \,\subsetneq \,L^p \,\subsetneq \,L^{p,\,q_2}
\end{equation}

\noindent We will make use of the following $L^{p,\,q}$ version of H\"older's inequality
\begin{equation}\label{h}
\left\|fg\right\|_{p, q} \leq \left\|f\right\|_{p_1, q_1} \left\|g\right\|_{p_2, q_2}\,\,\,\,
\end{equation}
where $\frac{1}{p}$ = $\frac{1}{p_1}$ + $\frac{1}{p_2}$ $<$ $1$,\,\,\,$\frac{1}{q}$ = $\frac{1}{q_1}$ + $\frac{1}{q_2}$ and $1$ $\leq$ $q_1, q_2$ $\leq$ $\infty$.

\section{Sobolev, Hardy, Hardy--Sobolev inequalities. Proof of Theorem\ref{th1}.}\label{sec:ckn}

The following proposition describes Sobolev improved within the framework of Lorentz spaces.

\vskip 0.3cm
\begin{proposition}\label{prop:importante}
Let $n \in \left[d,D\right]$. Suppose that $1< p<n$ and that $1\leq q\leq\infty$. Then
there exists a constant $c = c(n, p, q) > 0$ such that
\begin{equation}\label{eq:SL}
\left\|f\right\|_{\frac{np}{n-p},q} \leq c \left\|\left|X f\right|\right\|_{p,q}
\end{equation}
for all $f \in C^\infty _c (G)$.

In particular, if $p\leq q$ then 
\begin{equation}\label{eq:SLp}
\left\|f\right\|_{\frac{np}{n-p},q} \leq C \left\|\left|X f\right|\right\|_p
\end{equation}
for all $f \in C^\infty _c (G)$.

\end{proposition}

\begin{proof}
Recall that we have the pointwise estimate \eqref{eq:iné} where the potential operator $T$ is defined by convolution with the kernel $$N(x) = \frac{\rho (x)}{V(\rho (x))}.$$
But this kernel belongs to the space $L^{\frac{d}{d-1},\infty}\cap L^{\frac{D}{D-1},\infty}$. Actually (see also \cite{loh} and \cite{meda}),
using the growth conditions \eqref{D} and \eqref{d}, it is easy to see that for all $t >0$,
\begin{equation}\label{eq:N}
N^\ast (t)\leq min\left(\frac{1}{t^{\frac{d-1}{d}}},\frac{1}{t^{\frac{D-1}{D}}}\right),
\end{equation}
so$$\sup_{0<t<1} t^{\frac{d-1}{d}}N^\ast (t)+\sup_{t\geq1} t^{\frac{D-1}{D}}N^\ast (t)<\infty.$$
Using the notation of \cite{tycho}, it means that $$N\in L^{\infty,\infty}_{\frac{d}{d-1},\frac{D}{D-1}}.$$
But since $\frac{d}{d-1}>\frac{D}{D-1}$, then we have $$L^{\infty,\infty}_{\frac{d}{d-1},\frac{D}{D-1}} =L^{\frac{d}{d-1},\infty} \cap L^{\frac{D}{D-1},\infty},$$ and so $N \in L^{r,\infty}$ for all $r\in \left[\frac{D}{D-1},\frac{d}{d-1}\right]$.
Now by \cite[Lemma 4.8]{hunt}
 we have $$\left\|X_if \ast N\right\|_{b,q}\leq C \left\|X_if\right\|_{p,q} \left\|N\right\|_{\frac{n}{n-1},\infty},$$ with
$$0<\frac{1}{b}=\frac{1}{p}-\frac{1}{n}<1,$$
thus obtaining the inequality \eqref{eq:SL}.
For \eqref{eq:SLp}, we use \eqref{in} and \eqref{id} to conclude.

\end{proof}
\vskip 0.3cm
\begin{remark}\label{rq}
Using Proposition \ref{prop:importante}, we can immediatly obtain a Gagliardo-Nirenberg inequality under 
Lorentz norms: there exists $c>0$ such that

\begin{equation}\label{eq:gn}
\left\|f\right\|_{r,u}\leq c \left\|\left|X f\right|\right\|_{p,v}^a\left\|f\right\|_{t,w}^{1-a},
\end{equation}
with $\frac{1}{r}=\frac{a}{p_\ast}+\frac{1-a}{t}$,\,\,$\frac{1}{u}=\frac{a}{v}+\frac{1-a}{w}$ and $0\leq a\leq 1$.
Indeed, by a simple application of the H\"older-Lorentz inequality \eqref{h} we have
$$\left\|f\right\|_{r,u}\leq \left\|f\right\|_{p_\ast,v}^a\left\|f\right\|_{t,w}^{1-a},$$
we then apply Proposition \ref{prop:importante} to the first term on the right to conclude.

\end{remark}
\begin{remark}\label{rq badr}
A Haar measure of a Lie group $G$ of polynomial growth is obviously doubling
and it is well known that $G$ supports a Poincar\'e inequality $P_1$. Moreover for $n \in \left[d,D\right]$ we have 
the following global growth condition $\left|B\right|\geq cr^n$ for any ball $B$ of radius $r > 0$.
In \cite{badr}, Badr uses these three properties and starts from the following oscillation inequality
$$f^{\ast \ast}(t)-f^\ast(t)\leq c t^{1/n}\left|Xf\right|^{\ast \ast}(t)$$
to obtain first the  Gagliardo-Nirenberg inequality \eqref{rq gn} in its formulation for Lorentz spaces and
then inequality \eqref{eq:SLp} by taking $a=1$, $u=v=p$ and $r=p^\ast$.
Actually, resuming an approach due to Mario Milman and his collaborators (see e.g. \cite{mm}), Badr consider the more general situation of a complete Riemannian manifold with underlying doubling measure $\mu$, supporting a Poincar\'e inequality $(P_q)$ for some $1\leq q <\infty$
 and such that the following global growth condition $\mu(B)\geq cr^\sigma$
 holds for every ball $B$ of radius $r>0$ and for some $\sigma>q$.

\end{remark}

As a corollary we obtain the Hardy inequalities for the system $\mathbb{X}$ in Lorentz spaces.
\begin{corollary}
Let $1<p<d$ and $1\leq q\leq\infty$.Then
\begin{equation}\label{eq:H}
\left\|\rho^{-1}f\right\|_{p,q} \leq C \left\|\left|X f\right|\right\|_{p,q}
\end{equation}
for all $f \in C^\infty _c (G)$.
\end{corollary}
\begin{proof}
Let $B=B(e,1)$. We set $\rho_0^{-1}=\rho^{-1}\chi_B$ and $\rho_\infty^{-1}=\rho^{-1}-\rho_0^{-1}$.
By \eqref{D} and \eqref{d} we have $$\rho^{-1}=\rho_0^{-1}+\rho_\infty^{-1} \in L^{d,\infty}+L^{D,\infty}.$$ Then

\begin{align*}
\left\|\rho^{-1}f\right\|_{p,q} &\lesssim \left\|\rho_0^{-1}f\right\|_{p,q}+\left\|\rho_\infty^{-1}f\right\|_{p,q}
\\
&\lesssim\left\|\rho_0^{-1}\right\|_{d,\infty}\left\|f\right\|_{\frac{dp}{d-p},q}+\left\|\rho_\infty^{-1}\right\|_{D,\infty}\left\|f\right\|_{\frac{Dp}{D-p},q}
\\
&\lesssim \left\|\left|X f\right|\right\|_{p,q},
\end{align*}
where we have used \eqref{h} and Proposition \ref{prop:importante}.
\end{proof}

The following corollary is a Hardy--Sobolev type inequality in Lorentz spaces which interpolates \eqref{eq:SL} and \eqref{eq:H}.
\begin{corollary}\label{chs}
Let $n \in \left[d,D\right]$. Suppose that $1<p<d$ and $1\leq q\leq \infty$. Let $0\leq l\leq p$ and define $p^\ast(l):=\frac{p(n-l)}{n-p}$. Then
\begin{equation}\label{eq:HS}
\left\|\rho^{-l/p^\ast(l)}\,f\right\|_{p^\ast(l),q} \lesssim \left\|\left|X f\right|\right\|_{p,q}
\end{equation}
for all $f \in C^\infty _c (G)$.
\end{corollary}
\begin{proof}
If $l=0$ then $p^\ast(0)=p_\ast$ and \eqref {eq:HS} is just \eqref{eq:SL}. If $l=p$ then $p^\ast(p)=p$ and \eqref {eq:HS} is \eqref {eq:H}. Assuming $0<l<p$, the proof follows by noticing that $\rho_0^{-\frac{l}{p^\ast(l)}}\in L^{d \frac{p^\ast(l)}{l},\infty}$ and $\rho_\infty^{-\frac{l}{p^\ast(l)}}\in L^{D \frac{p^\ast(l)}{l},\infty}$. We then apply \eqref{h} to obtain

\begin{align*}
\left\|\rho^{-l/p^\ast(l)}\,f\right\|_{p^\ast(l),q}&\lesssim \left\|\rho_0^{-\frac{l}{p^\ast(l)}}\,f\right\|_{p^\ast(l),q}+\left\|\rho_\infty^{-\frac{l}{p^\ast(l)}}\,f\right\|_{p^\ast(l),q}
\\
&\lesssim \left\|\rho_0^{-\frac{l}{p^\ast(l)}}\right\|_{d\frac{p^\ast(l)}{l},\infty}\left\|f\right\|_{\frac{d}{d-l}p^\ast(l),q}+
\left\|\rho_\infty^{-\frac{l}{p^\ast(l)}}\right\|_{D\frac{p^\ast(l)}{l},\infty}\left\|f\right\|_{\frac{D}{D-l}p^\ast(l),q}
\\
&\lesssim \left\|\left|X f\right|\right\|_{p,q}
\end{align*}
where in the last inequality we have used Proposition \ref{prop:importante}, provided that $\frac{d}{d-l}p^\ast(l)$ and $\frac{D}{D-l}p^\ast(l)$ are Sobolev indices $\frac{Np}{N-p}$ and $\frac{Mp}{M-p}$ respectively  with $N$, $M \in \left[d,D\right]$. But it is easy to see that this is the
case since it suffices to take $N=\frac{pd(n-l)}{d(p-l)+l(n-p)}$, $M=\frac{pD(n-l)}{D(p-l)+l(n-p)}$ and to observe that $N\in \left[d,n\right]$ and $M\in \left[n,D\right]$.

\end{proof}

\begin{proof}[Proof of Theorem \ref{th1}]
Let us consider the case $a=1$, that is CKN inequalities without the interpolation term. By \eqref{eq:balance} and \eqref{delta}, we have $\frac{1}{r}=\frac{1}{p}-\frac{1+\gamma}{n}$ with $-1\leq\gamma\leq0$ so $r$ lies between $p$ and $p^\ast$: $r=tp+(1-t)p^\ast=\frac{p(n-tp)}{n-p}$ and since $\gamma$ = $-tp/r$ then $\rho^\gamma$ = $\rho^{-tp/p^\ast (tp)}$ and we conclude by Corollary \ref{chs}.

\noindent Let now $0$ $<$ $a$ $<$ $1$ 
and let the index $s$ defined by the relation
\begin{equation} \label{r}
\frac{1}{r}=\frac{a}{s}+\frac{1-a}{q}.
\end{equation}
Let also $q_1$, $q_2$ and $q_3$ such that $\frac{1}{q_3}=\frac{a}{q_2}+\frac{1-a}{q_1}$. Provided that
\begin{equation} \label{sig}
\gamma=a\sigma+(1-a)\beta,
\end{equation}
we obtain
\begin{align*}
\left\|\rho^\gamma f\right\|_{r,q_1}&=\left\|\rho^{a\sigma+(1-a)\beta}f^a f^{1-a}\right\|_{r,q_1}
\\
&\leq \left\|\rho^{a\sigma}f^a\right\|_{\frac{s}{a},\frac{q_2}{a}}\left\|\rho^{(1-a)\beta}f^{1-a}\right\|_{\frac{q}{1-a},\frac{q_3}{1-a}}
\\
&=\left\|\rho^\sigma f\right\|^a_{s,q_2}\left\|\rho^\beta f\right\|^{1-a}_{q,q_3}
\end{align*}
By \ref{eq:balance}, \ref{r} and \ref{sig}, we have 
\begin{equation} \label{eq:s}
\frac{1}{s}=\frac{1}{p}-\frac{1+\sigma}{n}.
\end{equation} \label{eq:s}

\noindent Since by \eqref{delta} we have $-1\leq \sigma \leq 0$, then the index $s$ lies between $p$ and $p^\ast$ so we conclude, as in the case $a = 1$, by using Corollary \ref{chs}.

\end{proof}
\section{Weighted $L^{p,\,q}$ estimates. Proof of Theorem \ref{th2}.}\label{sec:ckn1}
The following proposition gives sufficient conditions for the $L^{p,q}$ weighted boundedness of $T$. 
\begin{proposition}\label{pro}
Let $p$, $s$ $\geq$ $1$ and $1$ $\leq$ $q_2$ $\leq$ $q_3$ $\leq$ $\infty$. Let $u$ and $v$ two nonnonnegative locally integrable functions. Assume that
\begin{equation} \label{c1}
\sup_{t>0}\left(\int_t^\infty \tau^{q_3/s -1} \phi^{q_3}(\tau) v^{\ast q_3}(\tau)\,d\tau\right)^{1/q_3}\left(\int_0^t\tau^{-\acute{q_2}/p-1+\acute{q_2}}u^{\ast \acute{q_2}}(\tau)\,d\tau\right)^{1/\acute{q_2}}<\infty,
\end{equation}
and
\begin{equation} \label{c2}
\sup_{t>0}\left(\int_0^t \tau^{q_3/s -1}v^{\ast q_3}(\tau)\,d\tau\right)^{1/q_3}
\left(\int_t^\infty\tau^{-\acute{q_2}/p-1+\acute{q_2}}u^{\ast \acute{q_2}}(\tau)\phi^{\acute{q_2}}(\tau) \,d\tau\right)^{1/\acute{q_2}}<\infty,
\end{equation}
where $\frac{1}{q_2}$ + $\frac{1}{\acute{q_2}}$ = $1$ and  $\phi (t) = \min (t^{1/d -1},\,t^{1/D -1})$. Then
\[
\left\|vT(uf)\right\|_{s,q_{3}} \lesssim \left\|f\right\|_{p,q_{2}}.
\]
\end{proposition}
\begin{proof}
If we make an obvious variable changement and use \eqref{2}, then we obtain
\begin{align*}
\left\|vT(uf)\right\|_{s,q_{3}}^{q_{3}}&=\int_{0}^\infty \left[t^{1/s} \left(vT\left(uf\right)\right)^\ast \left(t\right)\right]^{q_{3}}\,\frac{dt}{t}
\\
&=2^{q_{3}/s}\int_{0}^\infty \left[t^{1/s} \left(vT\left(uf\right)\right)^\ast \left(2t\right)\right]^{q_{3}}\,\frac{dt}{t}
\\
&\leq 2^{q_{3}/s}\int_{0}^\infty \left[t^{1/s} v^\ast (t)\left(T\left(uf\right)\right)^\ast \left(t\right)\right]^{q_{3}}\,\frac{dt}{t}.
\end{align*}
We write $E(t)$ for the expression into the brackets. By \eqref{1} and a lemma due to O'Neil \cite[Lemma 1.5]{o'} (see also \cite{Y}), we obtain
\begin{align*}
E(t)&\leq t^{1/s} v^\ast (t)\left(T\left(uf\right)\right)^{\ast \ast}\left(t\right)
\\
&\leq t^{1+1/s} v^\ast \left(t\right)\left(uf\right)^{\ast \ast}\left(t\right)N^{\ast \ast}\left(t\right)+
t^{1/s} v^\ast (t) \int_{t}^\infty (uf)^{\ast \ast}(\xi)N^{\ast}(\xi ) \,d\xi.
\end{align*}
Using \eqref{3} we have
\begin{multline*}
E(t)\leq t^{1/s} v^\ast (t)N^{\ast \ast}(t)\int_{0}^{t}
u^\ast\left(\tau\right)f^\ast\left(\tau\right)\,d\tau\\
{}+t^{1/s} v^\ast (t)\int_t^\infty\left(\int_0^\xi u^\ast(\tau)f^\ast(\tau)\,d\tau\right)N^\ast(\xi)\,\frac{d\xi}{\xi}.
\end{multline*}
Changing the order of integration, the second term in this last sum equals
\begin{multline*}
=t^{1/s} v^\ast (t)\int_{0}^{t}
u^\ast(\tau)f^\ast(\tau)\,d\tau \left(\int_t^\infty N^\ast(\xi)\,\frac{d\xi}{\xi}\right)\\
{}+t^{1/s} v^\ast (t)\int_t^\infty u^\ast(\tau)f^\ast(\tau)\left(\int_\tau^\infty N^\ast(\xi)\,\frac{d\xi}{\xi}\right)\,d\tau.
\end{multline*}
Let us write $\psi (t) = \int_t^\infty N^\ast(\xi)\,\frac{d\,\xi}{\xi}$ so that
\begin{multline*}
E(t)\leq t^{1/s} v^\ast (t) (N^{\ast \ast}(t) + \psi (t)) \int_{0}^{t}
u^\ast\left(\tau\right)f^\ast\left(\tau\right)\,d\tau + t^{1/s} v^\ast (t) \int_t^\infty u^\ast(\tau)f^\ast(\tau) \psi (\tau)\,d\tau.
\end{multline*}
By \eqref{eq:N} we have
\begin{equation} \label{eq:!}
N^{\ast \ast}(t) \lesssim \phi (t),\,\,\,\,\,\,\psi (t) \lesssim \phi (t).
\end{equation}
Using \eqref{eq:!}, we obtain
\[
E(t)\lesssim t^{1/s} \phi (t) v^\ast(t)\int_0^tu^\ast(\tau)f^\ast(\tau)\,d\tau+t^{1/s}v^\ast(t)\int_t^\infty u^\ast(\tau)f^\ast(\tau) \phi (\tau)\,d\tau.
\]

Consequently,

\begin{multline}\label{eq:m}
\left\|vT(uf)\right\|_{s,q_3}\lesssim 
\left( \int_0^\infty\left[t^{1/s} \phi (t) v^\ast(t)\int_0^tu^\ast(\tau)f^\ast(\tau)\,d\tau\right]^{q_3}\,\frac{dt}{t}\right)^{1/q_3}\\
+\left( \int_0^\infty\left[t^{1/s}v^\ast (t)\int_t^\infty u^\ast(\tau)f^\ast(\tau) \,\phi (\tau)\,d\tau \right]^{q_3}\,\frac{dt}{t} \right)^{1/q_3}.
\end{multline}

Now we need the following theorems in \cite{bradly}: 
\begin{theorem}(\cite[Theorem 1]{bradly})
Let $1$ $\leq$ $p$ $\leq$ $q$ $\leq$ $\infty$. Suppose $u$ and $v$ are non-negative. Then
\[\left(\int_0^\infty\left[u(x)\int_0^xf(t)\,dt\right]^q\,dx\right)^{1/q}\leq c\left(\int_0^\infty\left[f(x)v(x)\right]^p\,dx\right)^{1/p},
\]

holds for non-negative $f$ if and only if
\[
\sup_{r>0}\left(\int_r^\infty u(x)^q\,dx\right)^{1/q}\left(\int_0^r v(x)^{-p{'}}\,dx\right)^{1/p{'}}<\infty
\]

\end{theorem}

\begin{theorem}(\cite[Theorem 2]{bradly})
Suppose that $1$ $\leq$ $p$ $\leq$ $q$ $\leq$ $\infty$ and that $u$ and $v$ are non-negative. Then
\[\left(\int_0^\infty\left[u(x)\int_x^\infty f(t)\,dt\right]^q\,dx\right)^{1/q}\leq c\left(\int_0^\infty\left[f(x)v(x)\right]^p\,dx\right)^{1/p},
\]
holds for non-negative $f$ if and only if
\[
\sup_{r>0}\left(\int_0^r u(x)^q\,dx\right)^{1/q}\left(\int_r^\infty v(x)^{-p{'}}\,dx\right)^{1/p{'}}<\infty.
\]
\end{theorem}

Then, using \cite[Theorem 1]{bradly} for the first term of the right-hand side of \eqref{eq:m}, we obtain:

\begin{align*}
\left( \int_0^\infty\left[t^{1/s} \phi (t) v^\ast(t)\int_0^tu^\ast(\tau)f^\ast(\tau)\,d\tau\right]^{q_3}\,\frac{dt}{t}\right)^{1/q_3}&\leq C\left( \int_0^\infty\left[t^{1/p-1/q_2}f^\ast(t)\right]^{q_2}\,dt \right)^{1/q_2}
\\
&=C\left\|f\right\|_{p,q_2}
\end{align*}

holds if and only if we have \eqref{c1}.
Similarly, using \cite[Theorem 2]{bradly} for the second term of the right-hand side of \eqref{eq:m}, we get 

\begin{align*}
\left( \int_0^\infty\left[t^{1/s}v^\ast (t)\int_t^\infty u^\ast(\tau)f^\ast(\tau) \,\phi (\tau)\,d\tau \right]^{q_3}\,\frac{dt}{t} \right)^{1/q_3}&\leq C\left( \int_0^\infty\left[t^{1/p-1/q_2}f^\ast(t)\right]^{q_2}\,dt \right)^{1/q_2}
\\
&=C\left\|f\right\|_{p,q_2}
\end{align*}
holds if and only if we have
 \eqref{c2}.
\end{proof}

Now if we apply the proposition to our particular weights $u = \rho^{-{\alpha}}$ and $v = \rho^{-{\sigma}}$ with ${\alpha}$,\,$\sigma$ $> 0$, then we obtain
\begin{corollary}\label{cor}
Let $1$ $<$ $p$ $\leq$ $s$ $<$ $\infty$ and $\alpha$, $\sigma$ $>$ $0$. Assume that
\begin{align}
1- \frac{d}{p} &< \alpha < \frac{d}{p'}, \label{E:a}\\
\sigma &< \frac{d}{s}, \label{E:b}\\
\alpha + \sigma - 1 &= \frac{d}{s} - \frac{d}{p} \label{E:c}\\
\frac{1}{s} - \frac{1}{p} &\leq \frac{\alpha + \sigma -1}{D}. \label{E:d}
\end{align}
Then,
\begin{equation}\label{eq:int}
\left\| \rho^{-{\sigma}} T(f)\right\|_{s,q_3} \lesssim \left\| \rho^{{\alpha}} f\right\|_{p,q_2},
\end{equation}
where $1 \leq q_2 \leq q_3 \leq \infty$.
\end{corollary}
\begin{proof}
The second inequality in \eqref{E:a} and \eqref{E:d} leade 
\begin{equation}\label{E:e}
\frac{1}{s} - \frac{\sigma -1}{D}-1 < 0,
\end{equation}
and the first inequality in \eqref{E:a} leads
\begin{equation}\label{E:f}
\alpha > 1- \frac{D}{p}.
\end{equation}
With these six conditions, we obtain \eqref{c1} and \eqref{c2}.
\end{proof}
\vskip 1cm

\begin{proof}[Proof of Theorem \ref{th2}]
Now the proof is a simple application of Corollary \ref{cor}.
Let $0$ $<$ $a$ $\leq$ $1$ and take $1\leq s,q \leq +\infty$ such that

$$\frac{1}{r}=\frac{a}{s}+\frac{1-a}{q}.$$
As in the proof of Theorem \ref{th1}, using \eqref{h} we obtain
$$\left\| \rho^{-\gamma} f\right\|_{r, q_1} \leq 
\left\|\rho^{-\sigma} f\right\|^a_{s, q_2} \left\|\rho^{-\beta} f\right\|_{q,q_3}^{1-a}$$
provided $\gamma=a\sigma+(1-a)\beta$ and $\frac{1}{q_1}=\frac{a}{q_2}+\frac{1-a}{q_3}$.
\noindent But by \eqref{eq:iné} and Corollary \ref{cor} we obtain
\begin{align}
\left\| \rho^{-\sigma} f\right\|_{s, q_2}&\lesssim \left\|\rho^{-\sigma} T(\left| Xf \right|)\right\|_{s, q_2} \\
&\lesssim \left\|\rho^\alpha \left| Xf \right|\right\|_{p, q_2}
\end{align}
under the conditions \eqref{E:a}, \eqref{E:b}, \eqref{E:c} and \eqref{E:d}, thus obtaining
$$\left\| \rho^{-\gamma} f\right\|_{r, q_1} \lesssim 
\left\|\rho^\alpha \left| Xf \right|\right\|^a_{p, q_2} \left\|\rho^{-\beta} f\right\|_{q,q_3}^{1-a}.$$
Now we proceed exactly as D'Ancona and Luca' in \cite{luca} to
rewrite this set of conditions in a compact form, eliminating the parameters $\sigma$ and $s$, so we omit the details.
\end{proof}

%%%%%%%%%%%%%%   End of the text   %%%%%%%%%%%%%%%%%%%%%%%%%%%%%%%%%%%%%%%%%%%%%%%%%%%%

%%%%%%%%%%%%%%   Bibliography   %%%%%%%%%%%%%%%%%%%%%%%%%%%%%%%%%%%%%%%%%%%%%%%%%%%%%%%

%%%%%%%%%%%%%%%%%%%%%%%%%%%%%%%%%%%%%%%%%%%%%%%%%%%%%%%%%%%%%%%%%%%%%%%%%%%%%%%%%%%%%%%

%%%%%%%%%%%%%%   End of the document   %%%%%%%%%%%%%%%%%%%%%%%%%%%%%%%%%%%%%%%%%%%%%%%%
\end{document}